\documentclass[11pt]{article}
\usepackage{amsmath}
\usepackage{amssymb}
\usepackage{amsthm}
\usepackage[ansinew]{inputenc}

\usepackage[all]{xy}
\usepackage{url}
\usepackage{graphicx}

\usepackage{verbatim}

\def\color[#1]#2{}

\makeatletter \gdef\th@break{\normalfont\slshape
\def\@begintheorem##1##2{\item[%
    \rlap{\vbox{\hbox{\hskip \labelsep\theorem@headerfont ##1\ ##2}%
                \hbox{\strut}}}]}%
\def\@opargbegintheorem##1##2##3{%
\item[\rlap{\vbox{\hbox{\hskip \labelsep \theorem@headerfont
                  ##1\ ##2\ ##3}%
                 \hbox{\strut}}}]}}
\makeatother

\newtheorem{theorem}{Theorem}[section]

\newtheorem{definition}{Definition}[section]
\newtheorem{proposition}{Proposition}[section]
\newtheorem{corollary}{Corollary}[section]
\newtheorem{remark}{Remark}

\newtheorem{lemma}{Lemma}[section]

\def\F{\mathbb{F}}

\def\Z{\mathbb{Z}}

\def\P{\mathbb{P}}

\def\GF{\mathbb{F}}

\def\Jac{\operatorname{Jac}}
\def\Aut{\operatorname{Aut}}

\def\AS{\operatorname{AS}}

\def\tr{\operatorname{tr}}

\def\Gal{\operatorname{Gal}}

\def\sgn{\operatorname{sgn}}

\title{Genus $3$ curves with many involutions and application to maximal curves in characteristic $2$}
\author{Enric Nart \& Christophe Ritzenthaler\thanks{The authors acknowledge support from the project MTM2006-11391 from the Spanish MEC}}

\begin{document}
\maketitle

\begin{abstract}

Let $k=\F_q$ be a finite field of characteristic $2$. A genus $3$ curve $C/k$ has many involutions if the group of $k$-automorphisms admits a $C_2\times C_2$ subgroup $H$ (not containing the hyperelliptic involution if $C$ is hyperelliptic). Then $C$ is an Artin-Schreier cover of the three elliptic curves obtained as the quotient of $C$ by the nontrivial involutions of $H$, and the Jacobian of $C$ is $k$-isogenous to the product of these three elliptic curves. 
In this paper we exhibit explicit models for genus $3$ curves with many involutions, and we compute explicit equations for the elliptic quotients. We then characterize when a triple $(E_1,E_2,E_3)$ of elliptic curves admits an Artin-Schreier cover by a genus $3$ curve, and we apply this result to the construction of maximal curves. As a consequence, when $q$ is nonsquare and $m:=\lfloor 2 \sqrt{q} \rfloor \equiv 1,5,7 \pmod{8}$, we obtain that $N_q(3)=1+q+3m$. We also show that this occurs for an  infinite number of values of $q$ nonsquare.
\end{abstract}

Let $C$ be a smooth, absolutely irreducible, projective curve of genus $g>0$ over a finite field $k=\GF_q$. The question to determine the maximal number of points $N_q(g)$ of such a curve $C$ is a tantalizing one. Curves such that $\# C(k)=N_q(g)$ are called \emph{maximal curves}.  Serre-Weil bound shows that $N_q(g) \leq 1+q+gm$, where $m=\lfloor 2 \sqrt{q} \rfloor$. However, no general formula is known for the value of $N_q(g)$ (so far not even for infinitely many values of $q$) when $g>2$ is fixed and $q$ is not a square. For $g=3$, because of the so-called Serre's twisting factor (or Serre's obstruction, see \cite{LR}, \cite{LRZ}), the best general result  is that for a given $q$, either $q+1+3m-N_q(3) \leq 3$ or $M_q(3)-(q+1-3m) \leq 3$, where $M_q(3)$ is the minimum number of points \cite{lauterg3}. Although this obstruction is now better understood and can be computed in some cases \cite{ritzen}, we are still not able to find $N_q(3)$ for a general $q$. However, when $q$ is a square, $N_q(3)$ is known for infinitely many values; see \cite{ibukiyama} when the characteristic is odd, and \cite{nrss} for the characteristic $2$ case, where $N_q(3)$ is determined for all square $q$.

In this article, we construct maximal genus $3$ curves over $\F_q$ for infinitely many values of $q=2^n$ nonsquare. Our result will be the easy consequence of the computation of the decomposition of the Jacobian of all genus $3$ curves with many involutions. In section 1 we show that the natural equivalence classes of pairs $(C,H)$, where $C$ is a genus $3$ curve and $H$ a $C_2\times C_2$ subgroup of $\Aut_k(C)$ (not containing the hyperelliptic involution if $C$ is hyperelliptic), are in bijection with the natural equivalence classes of Artin-Schreier covers of triples of elliptic curves (Definition \ref{ascover}). In particular, the Jacobian of a genus $3$ curve with many involutions is totally split. In section 2 we exhibit models of all $k$-isomorphism classes of genus 3 curves $C$ with many involutions, and we compute explicit equations for the three elliptic quotients. 
By retro-engineering, an appropriate choice of values for the parameters of our families enables us to characterize in section 3 the triples $(E_1,E_2,E_3)$ of elliptic curves that admit an Artin-Schreier cover by a genus $3$ curve. Theorem \ref{thhyp} deals with the hyperelliptic case and Theorems \ref{old}, \ref{reconstruction} with the non-hyperelliptic case. These results can be seen as an analogue of \cite[Sec.4]{HLP} in characteristic $2$. In section 4 we use these criteria to
construct maximal curves when $m\equiv 1,5,7 \pmod{8}$ (Corollary \ref{optimal}). We show that the case $m \equiv 1 \pmod{4}$ occurs infinitely often  (Lemma \ref{infinite}) and so we get an infinite family of values  of $N_q(g)$ for $g>2$ fixed and $q$ nonsquare.
In the case where $m \equiv 0,2,6 \pmod{8}$ we are able to show that $N_q(3) \geq q+1+3m-3$ and we give a sufficient condition for equality. In the other cases, $m \equiv 3,4 \pmod{8}$, the situation is more complicated and we could not get similar results. One may look at this dichotomy as another manifestation of Serre's obstruction.\bigskip

\noindent
{\bf Notations.} 
The field $k$ will be $\GF_{q}$, with $q=2^n$, $n \geq 1$. We denote the Artin-Schreier subgroup of $k$ by $$\AS(k):=\{x+x^2\,|\,x\in k\}=\ker(k\stackrel{\tr}\longrightarrow \GF_2),$$ and we fix once and for all an element $r_0\in k\setminus \AS(k)$  of trace $1$. If $q$ is nonsquare we take $r_0=1$.
We denote by $\sigma\in \Gal(\overline{k}/k)$
the Frobenius automorphism, $\sigma(x)=x^q$, which is a generator of this Galois group as a profinite group.

We denote $\lfloor s \rfloor$ the integer part of the real $s$ and $\{s\}$ its fractional part.  A curve will always mean a smooth, projective and absolutely irreducible curve.

\section{Curves with many involutions and Artin-Schreier covers}
\begin{definition}
A genus $3$ curve $C$ over $k$ is said to have \emph{many involutions} if $\Aut_k(C)$ admits a subgroup $H$ isomorphic to $C_2\times C_2$ and not containing the hyperelliptic involution, if $C$ is hyperelliptic.

Let $C,C'$ be genus $3$ curves with many involutions, with respective $C_2\times C_2$ subgroups $H,H'$. We say that the pairs $(C,H)$, $(C',H')$ are equivalent if there is a $k$-isomorphism $\varphi\colon C\to C'$ such that $\varphi H\varphi^{-1}=H'$.
\end{definition}

\begin{definition}\label{ascover}
An Artin-Schreier cover of a triple $(E_1,E_2,E_3)$ of elliptic curves over $k$ is a commutative diagram:
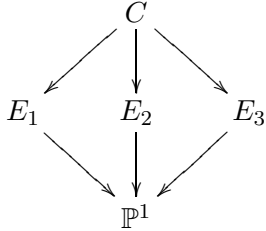
\begin{center}
\leavevmode
\xymatrix{
& C \ar[dl] \ar[d] \ar[dr] & \\ E_1 \ar[dr] & E_2 \ar[d]  & E_3 \ar[dl] \\ & \P^1 & }
\end{center}
where $C$ is a genus $3$ curve over $k$, and all maps are separable degree two (Artin-Schreier) morphisms defined over $k$. 
\end{definition}
There is a natural definition of equivalence of Artin-Schreier covers of triples of elliptic curves, whose formulation is left to the reader. 

The curve $C$ on the top of an Artin-Schreier cover has many involutions. In fact, $k(C)/k(\P^1)$ is a biquadratic extension with Galois group isomorphic to $C_2\times C_2$. The three nontrivial elements of this Galois group are the three nontrivial involutions of the quadratic extensions $k(C)/k(E_i)$, for $i=1,2,3$. Hence, $\Aut_k(C)$ admits a $C_2\times C_2$ subgroup too, and it does not contain the hyperelliptic involution (if $C$ were hyperelliptic), because the quotients of $C$ by these nontrivial involutions are elliptic curves. 

Conversely, any curve with many involutions arises in this way from an Artin-Schreier cover of three elliptic curves. 


\begin{proposition} \label{artins}
Let $C$ be a genus $3$ curve with many involutions, and $H=\{1,i_1,i_2,i_3\}$ a $C_2\times C_2$ subgroup of $\Aut_k(C)$, not containing the hyperelliptic involution if $C$ is hyperelliptic. Then, $C/H$ is isomorphic to $\P^1$, the curves $C/\langle i_s\rangle$ are elliptic curves, and the canonical maps $C\to C/\langle i_s\rangle\to C/H$, for $s=1,2,3$, determine an Artin-Schreier cover.
\end{proposition}

\begin{proof}
In all cases, $i_1,\,i_2.\,i_3$ have fixed points (cf. the remarks after the proofs of Propositions  \ref{modelshyp} and \ref{modelsnonhyp} below) so   that the respective quotients of $C$ by these involutions are three genus $1$ curves (it cannot be genus $0$ curves since these involutions are not the hyperelliptic one). Since $k$ is finite they are  elliptic curves $E_1$, $E_2$, $E_3$, over $k$ . Using \cite[Thm.B]{kani} with respect to the group $H$ we  get that
$$\Jac(C)^2 \times \Jac(C/H)^4 \sim \Jac(E_1)^2 \times \Jac(E_2)^2 \times \Jac(E_3)^2.$$ 
Hence by dimension count, $C/H$ is of genus $0$ and again since $k$ is finite we have $C/H \simeq \P^1$.
Finally, the three composition morphisms $C\to E_s\to C/H\simeq \P^1$ coincide with the canonical quotient map $C\to C/H$, so that they determine an Artin-Schreier cover.
\end{proof}

\begin{corollary}
There is a natural bijective correspondence between equivalence classes of pairs $(C,H)$ of curves with many involutions, and equivalence classes of Artin-Schreier covers of triples of elliptic curves.
\end{corollary}

In section 3 we shall determine what triples $E_1,E_2,E_3$ of elliptic curves over $k$ admit an Artin-Schreier cover (Theorems \ref{thhyp}, \ref{old} and \ref{reconstruction}). By Poincar\'e's complete reducibility theorem and the proof of Proposition \ref{artins}, we get
\begin{equation*} \label{isogeny}
\Jac(C)  \sim \Jac(E_1) \times \Jac(E_2) \times \Jac(E_3).
\end{equation*} 
Thus, by an appropriate choice of these elliptic curves, the genus $3$ curve covering them will be a maximal curve.

\section{Elliptic quotients of curves with many involutions}

\subsection{Models of curves of genus $3$ with many involutions}

The following two propositions are extracted from the results of \cite[Sec.3]{ns} in the hyperelliptic case, and from those of \cite[Sec.1.4]{nr} and \cite[Sec.3]{nrss} in the non-hyperelliptic case.

\begin{proposition} \label{modelshyp}
Let $C$ be a hyperelliptic genus $3$ curve over $k$, with many involutions. Then, $C$  is ordinary and it is isomorphic over $k$ to a curve in one of these two families:

$$(\operatorname{Hyp}_{a})\qquad \quad C_{a,r,t}\colon\quad y^2+y=a\left(x+\frac{t}{x}\right)+a(t+1) \left(\frac{1}{x+1}+\frac{t}{x+t}\right)+r,\qquad \qquad \qquad $$
where $a,\,t \in k^*$, $t\ne 1$, and $r\in\{0,r_0\}$. These curves have involutions
$$
i_1(x,y)=\left(\frac{t}{x},y\right), \quad i_2(x,y)=\left(\frac{x+t}{x+1},y\right), \quad i_3(x,y)=\left(\frac{tx+t}{x+t},y\right).$$
$$
(\operatorname{Hyp}_{\,b})\qquad \ \qquad C_{b,r,s,t}\colon\quad y^2+y=b\left(\frac{1}{x^2+x+s}+\frac{1}{x^2+x+t}\right)+r,\qquad \ \qquad \qquad \qquad $$
where $b,\,s,\,t \in k$, $b \ne 0$, $s,t\not\in AS(k)$, $s\ne t$, and $r\in\{0,r_0\}$. These curves have involutions
$$i_1(x,y)=(x+1,y), \quad i_2(x,y)=(x+u,y), \quad i_3(x,y)=(x+u+1,y),$$
where $u \in k$ satisfies $u^2+u=s+t$.\\

Moreover any pair $(C,H)$ of a hyperelliptic genus $3$ curve $C$ over $k$ with many involutions and a subgroup of $k$-automorphisms $H \simeq C_2 \times C_2$, not containing the hyperelliptic involution, is  equivalent to the pair  given by a curve of exactly one of the two families $\operatorname{Hyp}_{a}$ or $\operatorname{Hyp}_{\,b}$ and the subgroup generated by $i_1,i_2,i_3$.
\end{proposition}
\begin{proof}
Only the last claim needs some explanations. For any $C\in \operatorname{Hyp}_{a}\cup \operatorname{Hyp}_{\,b}$, the group $H=\{1,i_1,i_2,i_3\}$ is the only $C_2\times C_2$ subgroup in $\Aut(C)$ not containing the hyperelliptic involution. Moreover, no curve in the family $\operatorname{Hyp}_{\,a}$ is isomorphic over $k$ to  a curve in the family $\operatorname{Hyp}_{\,b}$.
\end{proof}

\begin{remark}
The fixed points of $i_1,i_2,i_3$ always coincide; for the family $\operatorname{Hyp}_{\,a}$ they are $\{(\sqrt{t},y),(\sqrt{t},y+1)\}$, with $y^2+y=a(t+1)$, whereas for the family $\operatorname{Hyp}_{\,b}$
they are the two points at infinity.
\end{remark}

\begin{proposition} \label{modelsnonhyp}
Let $C$ be a non-hyperelliptic genus $3$ curve over $k$,  with many involutions. Then, $C$  is either supersingular and  isomorphic over $k$ to a plane quartic in the family SS, or it is ordinary and isomorphic over $k$ to a  plane quartic in one of the two families $\operatorname{NHyp}_a$ or $\operatorname{NHyp}_b$ below. 
$$(\operatorname{SS})\qquad \qquad \qquad C_{d,e,f,g}\colon\quad \ y^4 + f y^2 z^2 + g y z^3 = x^3 z+d x^2  z^2 + e x^4,\qquad \qquad \qquad\qquad \ $$with $d,\,e,\,f,\,g\in k$, $g \ne 0$, and the equation $y^3+f y+g=0$ has three roots $v_1,\,v_2,\,v_3$ in $k$. These curves have involutions 
$$
i_1(x,y,z)=(x,y+v_1,z),\quad 
i_2(x,y,z)=(x,y+v_2,z),\quad 
i_3(x,y,z)=(x,y+v_3,z).
$$
$$(\operatorname{NHyp}_a)\qquad C_{a,c,e,r}\colon\quad (a(x^2+y^2)+c z^2+ x y+ e z(x+y))^2=(r(x^2+y^2)+xy)z(x+y+z),\quad$$
where $a,c,e\in k$, $r\in\{0,r_0\}$, $c \ne 0$, $a \ne r, r+a+e+c \ne
0$. These curves have involutions
$$i_1(x,y,z)=(y,x,z), \quad i_2(x,y,z)=(x+z,y+z,z), \quad i_3(x,y,z)=(y+z,x+z,z).$$
$$(\operatorname{NHyp}_b)\qquad C_{a,c,d,r}\colon\quad (a(x^2+y^2)+c z (x+y+z) + d x y)^2=(r(x^2+y^2)+xy)z(x+y+z),\quad$$
where $a,c,d \in k$, $r\in\{0,r_0\}$, $cd \ne 0$, $c+ d \ne 1$, $a+dr \ne
0$. These curves have  involutions 
$$i_1(x,y,z)=(y,x,z), \quad i_2(x,y,z)=(x,y,x+y+z),
\quad i_3(x,y,z)=(y,x,x+y+z).$$
Moreover any pair $(C,H)$ of an ordinary non hyperelliptic genus $3$ curve $C$ over $k$ with many involutions and a subgroup of $k$-automorphisms $H \simeq C_2 \times C_2$, is  equivalent to the pair given by a curve of exactly one of the two families $(\operatorname{NHyp}_a)$ or $(\operatorname{NHyp}_b)$ and the subgroup generated by $i_1,i_2,i_3$.
\end{proposition}
\begin{proof}
Again, only the last claim requires some precisions. Let $H_a$, $H_b$ be the $C_2\times C_2$ subgroups generated by the involutions $i_1,i_2,i_3$ of the curves respectively in $\operatorname{NHyp}_a$ and $\operatorname{NHyp}_b$. Now there are curves in $\operatorname{NHyp}_a\cap \operatorname{NHyp}_b$, given by  equations of the type: $$(a(x^2+y^2)+c z (x+y+z) + x y)^2=(r(x^2+y^2)+xy)z(x+y+z),$$
with $c\ne0$ and $a\ne r$. For these curves, $\Aut_k(C)$ contains the $D_8$ subgroup generated by $H_a\cup H_b$; among them, Klein's quartic ($c=a=1,\,r=0$) has the larger group of automorphisms: $\Aut_k(C)=PGL_3(\GF_2)$. Thus, these curves have different $C_2\times C_2$ subgroups inside $\Aut_k(C)$; however, all these subgroups fall into only two conjugacy classes, represented by $H_a$ and $H_b$. Hence, these curves determine only two equivalence classes of Artin-Schreier covers, represented by the pairs $(C,H_a)$ and $(C,H_b)$. Summing up, if we think in terms of equivalence of pairs $(C,H)$ then the families $\operatorname{NHyp}_a$ and $\operatorname{NHyp}_b$ have no intersection.
\end{proof}

\begin{remark}
The nontrivial involutions in $H_a$ have pairwise disjoint $2$-sets of fixed points on any curve in the family $\operatorname{NHyp}_a$. The nontrivial involutions in $H_b$ have the same $2$-set of fixed points on any curve in the family $\operatorname{NHyp}_b$; it is the set $\{(x,x,1),(x+1,x+1,1)\}$, for $x^2+x=cd$.
\end{remark}

\subsection{Ordinary elliptic curves in characteristic $2$}

Let us review some well-known facts on ordinary elliptic curves over finite fields of characteristic $2$. The first result can be easily deduced from \cite[Appendix A]{silverman}.

\begin{lemma} \label{etwist}
Let  $E$ be an elliptic curve over $k$ with $j$-invariant $j_E$. Then, $E$ 
is ordinary if and only if $j_E\ne0$. In this case, there is a unique element $\, \sgn(E)\in\{0,r_0\}$ such that $E$ is $k$-isomorphic to the curve with Weierstrass equation $y^2+xy=x^3+ \sgn(E) x^2 + (j_E)^{-1}$. We call this discrete invariant $\sgn(E)$ 
the \emph{signature} of $E$. Two curves with the same $j$-invariant and different signature are quadratic twist of each other.
\end{lemma}

\begin{lemma} \label{8point}
An ordinary elliptic curve $E$ with $\sgn(E)=0$ has always a rational $4$-torsion point. Moreover, it has a rational $8$-torsion point if and only if $\tr(1/j_E)=0$.
\end{lemma}

\begin{proof}
Let us denote $a=1/j_E$. The non-trivial $2$-torsion point of $E$ is $(0,a^{1/2})$. For any
$Q=(x,y)\in E(\overline{k})$ with $x\ne0$, the $x$-coordinate of $2Q$ is $x_{2Q}=x^2+ax^{-2}$, so that $(a^{1/4},a^{1/2})$ and $(a^{1/4},a^{1/2}+a^{1/4})$ are $4$-torsion points.

We characterize now the rationality of half of a point on $E(k)$.  Let $P=(u,v)$ be a rational point on $E(k)$, with $u\ne 0$. The point $Q=(x,y)$ satisfies $2Q=P$ if and only if
\begin{equation} \label{half}
x^2+\frac{a}{x^2}=u, \quad y^2+xy=x^3+a
\end{equation}
has a solution in $k$.
The first equation has a solution $x$ in $k$ if and only if $au^{-2} \in \AS(k)$. Assume this is the case; then, the second equation has a solution $y \in k$ if and only if $(x^3+a)x^{-2} \in \AS(k)$. But
$$(x^3+a)x^{-2}=x+ax^{-2}=x+x^2+u\ \in\ u+\AS(k).$$
Thus, the system \eqref{half} has a rational solution if and only if $u,\,au^{-2}\in\AS(k)$.  On the other hand, since $P$ is a rational point on $E$, $y^2+uy+(u^3+a)=0$ has a solution in $k$ so $(u^3+a) u^{-2}=u+a u^{-2} \in \AS(k)$. Hence, $Q$ is rational if and only if $u\in\AS(k)$.

If we apply this to the $4$-torsion points, with $x$-coordinate $u=a^{1/4}$, we get a rational $8$-torsion point on $E$ if and only if $\tr(a^{1/4})=0$, or equivalently $\tr(a)=0$.
\end{proof}

Recall that, for any elliptic curve $E$ over $k$, the number of rational points is: $\#E(k)=q+1-\tr(E)$, where $\tr(E)\in\Z$ is the trace of the Frobenius endomorphism. The above lemma yields some information on the value of $\tr(E)$ modulo $8$. 

\begin{corollary}\label{trE}
Let $E$ be an ordinary elliptic curve with $\sgn(E)=0$. Then, if $q>2$, one has $\tr(E) \equiv 1 \pmod{4}$. Moreover, if $q>4$ then $\tr(E) \equiv 1 \pmod{8}$ if and only if $\tr(1/j_E)=0$. 
\end{corollary}

\noindent{\bf Remark. }Since the twisted elliptic curves have opposite trace, Corollary \ref{trE} provides analogous information for the trace of the curves with $\sgn(E)=r_0$.\bigskip

Finally, we recall a criterion that relates the signature of two ordinary elliptic curves  in terms of a given isomorphism as curves of genus one, defined over the quadratic extension $k_2$ of $k$.

Let $E$ be an ordinary elliptic curve defined by a Weierstrass equation 
$$
y^2+xy=x^3+rx^2+a, \quad r\in k,\ a\in k^*.
$$Let $N=(0,a^{1/2})$ be the unique nontrivial $2$-torsion point of $E$. Multiplication by $-1$ is given by the involution $i(x,y)=(x,y+x)$. Let $\Aut_{g=1}(E)$ be the group of geometric automorphisms of $E$ as a curve of genus one:
$$
\Aut_{g=1}(E)\simeq \{1,\,i\}\rtimes E(\overline{k}),
$$
and, $i\circ \tau_P=\tau_{-P}\circ i$, for all $P\in E(\overline{k})$, where $\tau_P\in \Aut_{g=1}(E)$ is the translation by $P$.

The reader may easily check that 
\begin{equation}\label{tauN}
\tau_N(x,y)=\left(\dfrac{a^{1/2}}x,\dfrac{a^{1/2}y}{x^2}+a^{1/2}+\dfrac {a^{1/2}}x+\dfrac a{x^2}\right).
\end{equation}

\begin{lemma}\label{twist}
Let $E/k$ be an ordinary elliptic curve, $F/k$ a curve of genus one, and $\phi\colon F\rightarrow E$ a $k_2$-isomorphism such that $\rho:=\phi^{\sigma}\phi^{-1}\in\Aut_{g=1}(E)$ is defined over $k$. Then, $F$ is $k$-isomorphic to $E$ if and only if $\rho=1$ or $\rho=\tau_N$, where $N$ is the non-trivial $2$-torsion point of $E$.  
\end{lemma}

\begin{proof}
Clearly $\rho^{\sigma}\rho=1$, and $\rho^{\sigma}=\rho$ by hypothesis; thus, $\rho$ is an involution. 
The twists of $E$ as a curve of genus one are parameterized by the pointed set $H^1(\Gal(\overline{k}/k),\Aut_{g=1}(E))$. A $1$-cocycle is determined by the choice of an automorphism, and the twist represented by $(F,\phi)$ corresponds to the $1$-cocycle determined by  $\rho=\phi^{\sigma}\phi^{-1}$.
Two automorphisms $\chi,\,\varphi$ determine the same twist if and only if there exists another automorphism $\psi$ such that: $\psi^{\sigma}\varphi\psi^{-1}=\chi$. In particular, $\rho$ determines the trivial twist if and only if $\rho=\psi^{\sigma}\psi^{-1}$ for some automorphism $\psi$. Now, both for $\psi=\tau_P$ and $\psi=\tau_P\circ i$, the automorphism $\psi^{\sigma}\psi^{-1}=\tau_{P^{\sigma}-P}$ is a translation. Thus, $\rho$ determines the trivial twist if and only if $\rho=\tau_{P^{\sigma}-P}$ for some $P\in E(\overline{k})$ such that $P^{\sigma}-P$ is a $2$-torsion point. This is equivalent to $\rho=1$ or $\rho=\tau_N$; in fact one can always find points $P$ such that $P^{\sigma}-P=0$ (a rational point) or $P^{\sigma}-P=N$ (an irrational halving of a rational point). 
\end{proof}

\subsection{Elliptic quotients of the curves in the family $\operatorname{Hyp}_a$}

\begin{proposition}\label{hypa}
For any curve $C=C_{a,r,t}$ in the family $\operatorname{Hyp}_a$, the Jacobian $\Jac(C)$ is $k$-isogenous to $E_1 \times E_2 \times E_3$, where
\begin{eqnarray*}
E_1 :& y^2+ xy  &=\ x^3+(r+a(t+1))\, x^2+ (a(t+1))^4 \\
E_2 :& y^2+ xy &=\ x^3 + (r+a(t+1))\, x^2+  (at)^4 \\
E_3 :& y^2+ xy &=\ x^3 + (r+a(t+1))\, x^2+ a^4.
\end{eqnarray*}
\end{proposition}
\begin{proof}

Let us compute first the quotient of $C$ by the involution $i_1$. The functions $X=x+\frac{t}{x}+t+1$, $Y=y$, are stable by $i_1$, and they lead to an Artin-Schreier model for the quotient curve:
$$
Y^2+Y=aX+a(t+1)^2\dfrac1X+r+a(t+1).
$$
Now, the change of variables $X=a^{-1}x$, $Y=(y+a^2(t+1)^2)/x$ establishes an isomorphism between this curve and $E_1$.

For the involution $i_2$, we use $X=\frac{x(x+t)}{x+1}$, $Y=y$, as invariant functions. The quotient curve admits an Artin-Schreier model: $Y^2+Y=aX+(at^2/X)+r+a(t+1)$,
which is isomorphic to $E_2$ via $X=a^{-1}x$, $Y=(y+a^2t^2)/x$.

For the involution $i_3$, we use $X=\frac{x(x+1)}{x+t}$, $Y=y$, as invariant functions. The quotient curve admits an Artin-Schreier model $Y^2+Y=aX+(a/X)+r+a(t+1)$, which is isomorphic to $E_3$
via $X=a^{-1}x$, $Y=(y+a^2)/x$.
\end{proof}

\subsection{Elliptic quotients of the curves in the family $\operatorname{Hyp}_b$}
\begin{proposition}\label{hypb}
For any curve $C=C_{b,r,s,t}$ in the family $\operatorname{Hyp}_b$, the Jacobian $\Jac(C)$ is $k$-isogenous to $E_1 \times E_2 \times E_3$, where
\begin{eqnarray*}
E_1 &: y^2+ xy  &= x^3+ r\, x^2+ b^4u^{-4}(u+1)^{-4} \\
E_2 &: y^2+ xy &= x^3 + (r+r_0)\, x^2+ b^4 u^4(u+1)^{-4} \\
E_3 &: y^2+ xy &= x^3 + (r+r_0)\, x^2+ b^4 u^{-4}(u+1)^4,
\end{eqnarray*}
where $u\in k$ satisfies $u(u+1)=s+t$.
\end{proposition}

\begin{proof}
The functions $X=x(x+1)$ and $Y=y$ are stable by $i_1$, and lead to the following Artin-Schreier model for the quotient curve:
\begin{equation}\label{model}
F\colon\quad Y^2+Y = b\left(\frac{1}{X+s}+\frac{1}{X+t}\right)+r=\dfrac b{s+t}\left(\frac{X+t}{X+s}+\frac{X+s}{X+t}\right)+r.
\end{equation}
Letting $c:=b/(s+t)$, the curve $F$ is $k$-isomorphic to $E_1$ via 
\begin{equation}\label{iso}
\phi\colon F\longrightarrow E_1,\quad \phi(X,Y)=\left(c\,\dfrac{X+t}{X+s},c\,\dfrac{X+t}{X+s}\,Y+c^2\right).
\end{equation}

For the involution $i_2$ we consider the invariant functions $X=x(x+u)$, $Y=y$, leading to an Artin-Schreier model:
$$F\colon \quad Y^2+Y= \frac{b(s+t)}{X^2+(u+1)X+st}+r=bu \left(\frac{1}{X+\alpha}+\frac{1}{X+\beta}\right)+r,$$
where $\alpha,\beta\in k_2$ are the roots of $X^2+(u+1)X+st=0$. They belong to $k$ if and only if $\tr(st/(u+1)^2)=0$, but this is never the case; indeed, 
$$\frac{st}{(u+1)^2}=\frac{s(u^2+u+s)}{u^2+1}=\frac{s^2+(u^2+1)s+(u+1)s}{u^2+1}=\left(\frac{s}{u+1}\right)^2+\frac{s}{u+1}+s,$$
so that $\tr(st/(u+1)^2)=\tr(s)=1$.

We are now back to the case of (\ref{model}), with $b,\,s,\,t$ replaced respectively by $bu, \, \alpha, \, \beta$. So, if we  now denote $c:=bu/(\alpha+\beta)=bu/(u+1)$, we get  as in (\ref{iso}) a $k_2$-isomorphism between $F$ and the elliptic curve $y^2+xy=x^3+rx^2+c^4$, which is the quadratic twist of $E_2$:
$$\phi\colon F\longrightarrow E'_2,\quad \phi(X,Y)=\left(c\,\dfrac{X+\beta}{X+\alpha},c\,\dfrac{X+\beta}{X+\alpha}\,Y+c^2\right).
$$
The automorphism $\rho:=\phi^{\sigma}\phi^{-1}$ of $E'_2$, as a curve of genus one, is the involution: $$\rho(x,y)=\left(\dfrac{c^2}x,\dfrac{c^2y}{x^2}+c^2+\dfrac{c^4}{x^2}\right).$$ Hence,  $\rho=-\tau_N$ (cf. (\ref{tauN})), and Lemma \ref{twist} shows that $F$ and $E'_2$ are not $k$-isomorphic as curves of genus one. Therefore, for any choice of a rational point of $F$ we obtain an elliptic curve necessarily $k$-isomorphic to $E_2$.

For the last involution, the same arguments work, just by substituting $u$ by $u+1$.
\end{proof}

\subsection{Elliptic quotients of the curves in the family SS}
This is taken directly from \cite[Sec.3]{nrss} where the decomposition type of the Jacobian of a supersingular curve of genus $3$ in characteristic $2$ was treated in full generality.

\begin{proposition}
For any curve $C=C_{d,e,f,g}$ in the family $\operatorname{SS}$, the Jacobian $\Jac(C)$ is $k$-isogenous to $E_1 \times E_2 \times E_3$, where
$$E_i : \quad y^2+ \frac{g}{v_i} y = x^3 + d x^2+ e,\quad i=1,2,3,$$
and $\,v_1,v_2,v_3$ are the three roots in $k$ of the equation $v^3+f v+g=0$. 
\end{proposition}

\subsection{Elliptic quotients of the curves in the family $\operatorname{NHyp}_a$}
\begin{proposition} \label{nhypa}
For any curve $C=C_{a,c,e,r}$ in the family $\operatorname{NHyp}_a$, the Jacobian $\Jac(C)$ is $k$-isogenous to $E_1 \times E_2 \times E_3$, where
\begin{eqnarray*}
E_1 &: y^2+ xy &= x^3 + e x^2+ (a+r)^2 (a+c+e+r)^2 \\
E_2 &: y^2+ xy &= x^3 + (e+r) x^2+ c^2 (a+c+e+r)^2 \\
E_3 &: y^2+ xy &= x^3 + (e+r) x^2+ c^2 (a+r)^2.
\end{eqnarray*}
\end{proposition}

\begin{proof}
We start with the quotient by the involution $i_1$. We work with the affine model of $C$ obtained by letting $z=1$. The functions $X=x+y$, $Y=xy$ are stable by $i_1$, and they lead to the following equation for the quotient curve
$$
Y^2 + XY + Y  = a^2 X^4+ r X^3+ (e^2+r) X^2 + c^2,
$$ 
which is isomorphic to $E_1$ via $X=\dfrac x{a+r}+1$, $Y=\dfrac{y+ex}{a+r}+\dfrac{ax^2}{(a+r)^2}+r$.

For the involution $i_2$, we start by a change of variable : $x
\leftarrow x+y$ so the involution becomes $i_2(x,y,z)=(x,y+z,z)$, and the equation of $C$ becomes
\begin{equation}\label{firsteq}
C\colon\quad a^2 x^4+ c^2 z^4+ (x^2+y^2)y^2+e^2 z^2 x^2 = (r x^2+y x
+y^2)z(x+z). 
\end{equation}
We work with the affine model of $C$ obtained by letting $x=1$. We choose then $Y=y(y+z)$, $Z=z$ and we obtain the
following equation for the quotient 
$$Y^2+ ZY+Y=c^2 Z^4+ (e^2+r) Z^2 + rZ+a^2,$$
which is isomorphic to $E_2$ via $Z=c^{-1}x+1$, $Y=c^{-1}y+c^{-1}x^2+c^{-1}ex+r$.

To deal with the third quotient we make the
change of variables $z \leftarrow x+y+z$. The curve $C=C_{a,c,e,r}$ becomes the curve $C'=C_{a+c+e,c,e,r}$ and the involution $i_3$ becomes $i'_2$. Therefore, the quotient curve is isomorphic to the elliptic curve obtained from $E_2$ by changing $a\leftarrow a+c+e$. And this is precisely $E_3$.
\end{proof}

\subsection{Elliptic quotients of the curves in the family $\operatorname{NHyp}_b$}
\begin{proposition}\label{nhypb}
For any curve $C=C_{a,c,d,r}$ in the family $\operatorname{NHyp}_b$, the Jacobian $\Jac(C)$ is $k$-isogenous to $E_1 \times E_2 \times E_3$, where
\begin{eqnarray*}
E_1 :& y^2+ xy &=\ x^3 + c^2d^2 x^2+ d^4 (a+dr)^4 \\
E_2 :& y^2+ xy &=\ x^3 + (c^2d^2+r) x^2+ c^4 (a+dr)^4 \\
E_3 :& y^2+ xy &=\ x^3 + (c^2d^2+r) x^2+ (c+d+1)^4 (a+dr)^4.
\end{eqnarray*}
\end{proposition} 

\begin{proof}
We start with the quotient by the involution $i_1$. We work with the affine model of $C$ obtained by letting $z=1$. The functions $X=x+y$, $Y=xy$ are stable by $i_1$, and they lead to the following Weierstrass equation for the quotient curve
$$
d^2Y^2 + XY + Y  = a^2 X^4+ r X^3+ (c^2+r) X^2 + c^2,
$$ 
which is isomorphic to $E_1$ via $X=\dfrac x{d(a+rd)}+1$, $Y=\dfrac{y}{d^3(a+rd)}+\dfrac{ax^2}{d^3(a+rd)^2}+r$.

For the involution $i_2$ we work with the affine model obtained by letting $y=1$. The functions $X=x$, $Z=z(x+z+1)$ are invariant and they yield the following model for the quotient curve:
$$
F\colon\quad c^2Z^2+rX^2Z+XZ+rZ=a^2X^4+d^2X^2+a^2.
$$
If $r=0$, the change of variables $X=x/ac$, $Z=(y+x^2)/ac^3$ sets a $k$-isomorphism beetwen this curve and $E_2$. However, if $r=r_0$ it is not easy to get rid of the term $rX^2Z$.
In this case we let $\alpha,\beta\in k_2$ be the roots of $x^2+x+r=0$, so that 
$(\alpha x+\beta)(\beta x+\alpha)=rx^2+x+r$. The involution of the plane
$$
I(X,Z)=\left( \dfrac{\alpha X+\beta}{\beta X+\alpha}, \dfrac Z{(\beta X+\alpha)^2}\right),
$$ 
sets a $k_2$-isomorphism between $F$ and the curve $F'$ with equation
$$
F'\colon\quad c^2Z^2+XZ=A^2X^4+d^2X^2+A^2,
$$
where $A=c(a+dr)$. This curve is $k$-isomorphic to the quadratic twist $E'_2$ of $E_2$ via
$$
\psi\colon F'\longrightarrow E'_2,\quad \psi(x,y)=(cAx,c^3Ay+c^2A^2x^2).
$$
We apply now Lemma \ref{twist} to the $k_2$-isomorphism $\phi=\psi  I\colon F\rightarrow E'_2$. Clearly, $I^{\sigma} I(x,y)=(x^{-1},yx^{-2})$, and straightforward computation shows that $\phi^{\sigma}\phi^{-1}=\psi I^{\sigma} I\psi^{-1}=-\tau_N$, 
where $N$ is the non-trivial $2$-torsion point of $E_2'$ (cf. (\ref{tauN}) for the explicit computation of $\tau_N$). Thus, Lemma \ref{twist} shows that $F$ is not $k$-isomorphic to $E'_2$, and it must be $k$-isomorphic to $E_2$.

To deal with the third quotient we make the
change of variables $x \leftarrow y+z$, $y \leftarrow x+z$. The curve $C=C_{a,c,d,r}$ becomes the curve $C'=C_{a,c+d+1,d,r}$ and the involution $i_3$ becomes $i'_2$. Therefore, the quotient curve is isomorphic to the elliptic curve 
$$
y^2+ xy \ =\ x^3 + ((c+d+1)^2d^2+r) x^2+ (c+d+1)^4 (a+dr)^4.
$$
obtained from $E_2$ by changing $c\leftarrow c+d+1$. This curve is isomorphic to $E_3$
via $y\leftarrow y+d^2x$.
\end{proof}

\section{Triples of elliptic curves admitting an Artin-Schreier cover}
We invert now the process of the previous section. Given a triple of elliptic curves, 
we determine when it is possible to reconstruct a genus $3$ curve with many involutions, having the given  curves as elliptic quotients.  

\begin{theorem}\label{thhyp} 
Let $(E_1,\,E_2,\,E_3)$ be a triple of ordinary elliptic curves over $k$, with $j$-invariants $j_1,j_2,j_3$. Then, $(E_1,\,E_2,\,E_3)$ admits an Artin-Schreier cover by a hyperelliptic genus $3$ curve if and only if 
\begin{equation}\label{inverses}
\frac{1}{j_1}+\frac{1}{j_2}+\frac{1}{j_3}=0. 
\end{equation}
\end{theorem}

\begin{proof}
Propositions \ref{hypa}, \ref{hypb} and the last point of Proposition \ref{modelshyp} show that condition (\ref{inverses}) is necessary. Conversely, suppose that (\ref{inverses}) is satisfied. In this case we have necessarily $q>2$. By reordering the indices we may assume that $\tr(E_2)\equiv\tr(E_3)\pmod{4}$. 

If $\tr(E_1)\equiv\tr(E_2)\equiv\tr(E_3)\pmod{4}$, we take a curve $C_{a,r,t}$ of the family 
$\operatorname{Hyp}_a$ with  
$$a=\left(\frac{1}{j_3}\right)^{1/4}, \quad t=\left(\frac{j_3}{j_2}\right)^{1/4}, \quad r=a(t+1)+\sgn(E_1).$$

If $\tr(E_1)\not\equiv\tr(E_2)\pmod{4}$, we take a curve $C_{b,r,s,t}$ of the family 
$\operatorname{Hyp}_b$ with  
$$b=\left(\frac{1}{j_2 j_3}\right)^{1/8}, \quad u=\left(\frac{j_1}{j_2}\right)^{1/8}, \quad r=\sgn(E_1),
$$
$s$ an arbitrary element in $k\setminus\AS(k)$ and $t=s+u+u^2$.
\end{proof}

For the sake of completeness we include the analogous result concerning the family SS, which was obtained in \cite[Thm.5.18]{nrss}.

\begin{theorem}\label{old} 
Let $(E_1,\,E_2,\,E_3)$ be a triple of supersingular elliptic curves over $k$. Then, if $q>64$, $(E_1,E_2,E_3)$ admits an Artin-Schreier cover by a non-hyperelliptic genus $3$ curve in the family SS. 
\end{theorem}

The applications to the existence of maximal curves when $q$ is nonsquare will be a consequence of the next result. 

\begin{theorem} \label{reconstruction}
Assume $q>2$. Let $(E_1,\,E_2,\,E_3)$ be a triple of ordinary elliptic curves with  $j$-invariant $j_1,j_2,j_3$, and denote $\sgn(E_1,E_2,E_3):=\sgn(E_1)+\sgn(E_2)+\sgn(E_3)\in\{0,r_0\}$. Consider the following elements in $k^*$:
$$
T_a:=\dfrac{(j_1+j_2+j_3)^2}{j_1j_2j_3},\qquad T_b:=\dfrac{j_1 j_2 j_3^2}{(j_1j_2+j_1j_3+j_2 j_3)^2}.
$$
Then, $(E_1,\,E_2,\,E_3)$ admits an Artin-Schreier cover by a non-hyperelliptic genus $3$ curve $C$ if and only if 
\begin{equation}\label{dim2'5}
T_a\in\sgn(E_1,E_2,E_3)+\AS(k),\ \mbox{ or }\ 
T_b\in\sgn(E_1,E_2,E_3)+\AS(k).
\end{equation}
\end{theorem}

\begin{proof}Propositions \ref{nhypa}, \ref{nhypb} and  the last point of Proposition \ref{modelsnonhyp} show that condition (\ref{dim2'5}) is necessary. Conversely, assume that (\ref{dim2'5}) is satisfied.
Let $s_i=(j_i)^{-1/4}$ for $i=1,2,3$.
We reorder the indices $1,2,3$ to have $\tr(E_2)\equiv\tr(E_3)\pmod{4}$, so that $\sgn(E_1)=\sgn(E_1,E_2,E_3)$.

Take $r=0$ if $\tr(E_1)\equiv\tr(E_2)\equiv\tr(E_3)\pmod{4}$, and $r=r_0$ otherwise. We want to show the existence of a curve $C_{a,c,e,r}$ in the family $\operatorname{NHyp}_a$, or a curve $C_{a,c,d,r}$ in the family $\operatorname{NHyp}_b$, satisfying respectively 
$$\begin{cases} (a+r) (a+r+e+c) &=\ s_1^2 \\
c (a+r+e+c) &=\ s_2^2 \\
c (a+r) &=\ s_3^2 \\
e\in\sgn(E_1)+\AS(k),&
\end{cases}\qquad
\begin{cases}
d (a+dr) &=\ s_1 \\
c (a+dr) &=\ s_2 \\
(1+c+d) (a+dr) &=\ s_3 \\
cd\in\sgn(E_1)+\AS(k).& \end{cases}
$$
These equations in the unknowns $a,c,d,e$ are easily solved:
$$\begin{cases}
a &=\ \frac{s_1 s_3}{s_2}+r \\
c &=\ \frac{s_3 s_2}{s_1} \\
e &=\ \frac{s_1 s_2}{s_3} + \frac{s_3 s_2}{s_1} +\frac{s_1 s_3}{s_2}\\
e&\in\ \sgn(E_1)+\AS(k). \end{cases}\qquad \qquad 
\begin{cases}
a &=\  s_1+s_2+s_3 +dr\\
c &=\ \frac{s_2}{s_1+s_2+s_3}\\d &=\ \frac{s_1}{s_1+s_2+s_3} \\
cd&\in\ \sgn(E_1)+\AS(k). \end{cases}$$
\end{proof}

\noindent{\bf Remarks. }
\begin {enumerate} \item In the non hyperelliptic case, the factors $T_a$, $T_b$ and condition (\ref{dim2'5}) reflect Serre's obstruction and they can be compared to the twisting factor $T$ and the ``to be a square" condition of \cite[Prop.15]{HLP}.
\item The expression of $T_b$ as a rational function of $j_1,j_2,j_3$ is not symmetric, but condition (\ref{dim2'5}) is symmetric in $j_1,j_2,j_3$, as the following identity shows:  
{\footnotesize $$
\dfrac{j_1j_2j_3^2}{(j_1j_2+j_1j_3+j_2j_3)^2}+\dfrac{j_1j_2^2j_3}{(j_1j_2+j_1j_3+j_2j_3)^2}=\dfrac{j_1(j_2+j_3)}{j_1j_2+j_1j_3+j_2j_3}+\dfrac{j_1^2(j_2+j_3)^2}{(j_1j_2+j_1j_3+j_2j_3)^2} \in \AS(k).
$$}
\end {enumerate}

\section{Application to maximal curves}

We are looking for genus $3$ curves with many points over a finite field $k$. The idea we use here is to look for ordinary elliptic curves $E$ such that the triple $E,E,E$ admits an Artin-Schreier cover by a genus $3$ curve $C$. Since $\Jac(C)$ is $k$-isogenous to $E\times E\times E$, for an adequate choice of the trace of $E$ the curve $C$ will be maximal. Since the hyperelliptic families define only a $2$-dimensional locus in the moduli space, the non hyperelliptic families are more suitable for our purpose. However for non hyperelliptic curves, Serre's precise version of Torelli theorem turns out to be a non trivial obstruction and we will only be able to  construct maximal curves for $\lfloor 2 \sqrt{q} \rfloor \not\equiv 3,4 \pmod{8}$.

\subsection{Some values of $N_q(3)$}
Let $m=\lfloor 2 \sqrt{q} \rfloor$. Serre-Weil bound shows that if $C$ is a genus $3$ curve over $k$ then $\#C(k) \leq q+1+3m$. We write $\#C(k) = q+1+3m-a$ with $a \geq 0$  called the \emph{defect} of the curve $C$. As usually we denote
$$N_q(3)=\sup_{C/k \; \textrm{of genus} \; 3}\{\#C(k)\}.$$
When $q$ is a square and $q > 16$, it was shown in \cite{nrss} that $N_q(3)=q+1+3m$. According to \cite{geer}, $N_2(3)=7=q+1+3m-2$, Therefore, we now concentrate on the case $q$ nonsquare, $q>2$.

\begin{theorem}
Suppose $q>2$ nonsquare, and let $m=\lfloor 2 \sqrt{q} \rfloor$. If $m \equiv 1,\,5,\,7 \pmod{8}$, there exists a genus $3$ curve $C$ over $k$ with defect $0$.

If $m \equiv 0,\,2,\,6 \pmod{8}$, there exists a genus $3$ curve $C$ over $k$ with defect $3$. 
\end{theorem}

\begin{proof}
Assume first $m\equiv1\pmod{4}$. Let $E$ be an ordinary elliptic curve over $k$ with trace $-m \equiv -1 \pmod{4}$, and let $j\in k^*$ be the $j$-invariant of $E$. We apply Theorem 
\ref{reconstruction} to $E_1=E_2=E_3=E$; we have $\sgn(E_1,E_2,E_3)=1$ and $T_b=1$, so that there exists a curve in the family $\operatorname{NHyp}_b$  such that $\Jac(C) \sim E^3$. This curve has defect $0$, because $\# C(k)=q+1-\tr(\Jac(C))=q+1+3m$.

For $m\equiv 2\pmod{4}$ we take an elliptic curve $E$ over $k$ with trace $-m+1\equiv-1\pmod{4}$, and the same argument shows the existence of a curve in the family $\operatorname{NHyp}_b$ with defect $3$:  $\# C(k)=q+1-\tr(\Jac(C))=q+1-3(1-m)$.

Suppose now $m\equiv-1\pmod8$. Take $E$ an elliptic curve with trace $-m\equiv1\pmod8$ and let $j\in k^*$ be the $j$-invariant of $E$. Corollary \ref{trE} show sthat $\tr(1/j)=0$. We apply Theorem \ref{reconstruction} to $E_1=E_2=E_3=E$; now  $\sgn(E_1,E_2,E_3)=0$ and $T_a=1/j$. so that there exists a curve in the family $\operatorname{NHyp}_a$  such that $\Jac(C) \sim E^3$. As we saw above, this curve has defect $0$.  

For $m\equiv 0\pmod{8}$ we take an elliptic curve $E$ over $k$ with trace $-m+1\equiv1\pmod{8}$, and the same argument shows the existence of a curve in the family $\operatorname{NHyp}_a$ with defect $3$.
\end{proof}

\noindent{\bf Remark. }
More explicitly, for the cases $m\equiv1,2\pmod4$ the curve
$$C : \left(j^{-1/4}(x^2+y^2)+z^2+xy+xz+yz\right)^2=xyz(x+y+z),$$
does the job. And for the cases $m\equiv 0,7\pmod8$ we can take the curve
$$C : \left( j^{-1/4}(x^2+y^2+z^2+xz+yz)+xy\right)^2=xyz(x+y+z).$$

\begin{corollary} \label{optimal}
Suppose $q>2$ nonsquare, and let $m=\lfloor 2 \sqrt{q} \rfloor$.
If $m \equiv 1,5,7 \pmod{8}$ then $N_q(3)=q+1+3m$. If $m \equiv 0, 2, 6 \pmod{8}$ and $\{2 \sqrt{q}\} < 1-4 \cos^2(3 \pi/7) \approx 0.8019$ then $N_q(3)=q+1+3m-3$.
\end{corollary}
\begin{proof}
We have only to deal with the cases $m \equiv 0,2,6 \pmod{8}$. We use the results of \cite{lauter} to prove that defects $0,1,2$ are not possible. Default $1$ is excluded \cite[Prop.2]{lauter}. For defect $0$, there would exist an elliptic curve with trace $m$. As $m$ is even, it means that this elliptic curve is supersingular. Its possible trace is then $0$ or $\pm \sqrt{2 q}$, and it cannot be equal to $m$. As for defect $2$, the same argument allows us to exclude the cases denoted $$(m,m,m-2),(m,m-1,m-1),(m,m+\sqrt{2}-1,m-\sqrt{2}-1),(m,m+\sqrt{3}-1,m-\sqrt{3}-1)$$
in \cite[Tab.1]{lauter}, because they imply the existence of a supersingular elliptic quotient. The case denoted $$(m-1,m+\frac{-1+\sqrt{5}}{2},m+\frac{-1-\sqrt{5}}{2})$$ can be excluded by the resultant $1$ method of \cite[Th.1a]{lauterhowe}. It remains the case denoted
$$(m+1-4 \cos^2 \frac{\pi}{7},m+1-4 \cos^2 \frac{2 \pi}{7},m+1-4 \cos^2 \frac{3 \pi}{7}).$$
Arguing as in \cite[2.1]{lauter}, the assumption on $\{2 \sqrt{q}\}$ excludes this case. This proves our result.
\end{proof}

It is not clear if we can get rid of the case
$$(m+1-4 \cos^2 \frac{\pi}{7},m+1-4 \cos^2 \frac{2 \pi}{7},m+1-4 \cos^2 \frac{3 \pi}{7})$$
for $q$ big enough. The isogeny class of abelian threefolds corresponding to this case contains a Jacobian at least for $q=2$. Moreover by \cite[Th.1.2]{howe}, there is always a principally polarized abelian variety in this absolutely simple class. Hence, whether or not it is a Jacobian depends only on Serre's twisting factor whose behavior is still quite unpredictable.

\noindent{\bf Remark. }
These methods yield also minimal curves for $m \equiv 1,\,3,\,7 \pmod{8}$.

\subsection{Infinitely many maximal curves}
In the even characteristic case, we proved in \cite{nrss} that there exists an maximal genus $3$ curve over $\F_q$ for all $q$ square, $q>16$. Actually, we proved that $N_q(3)=q+1+2\sqrt{q}$ for all square $q>16$ and $M_q(3)=q+1-2\sqrt{q}$ for all square $q>64$. In the odd characteristic case,
it was shown in \cite{ibukiyama} that for any odd prime number $p$, there is an infinite number of even degree extensions of $\GF_p$ admitting maximal genus 3 curves.

As far as we know, for odd degree extensions of prime fields (any characteristic) no such result is known for curves of genus $g>2$. The aim of this section is to show that Corollary \ref{optimal} applies for an infinite number of nonsquare $q$, leading to a result similar to that of Ibukiyama, for odd degree extensions of $\GF_2$.

\begin{lemma} \label{infinite}
There are infinitely many nonsquare $q$ such that $m \equiv 1 \pmod{4}$. 

There are infinitely many nonsquare $q$ such that $m \equiv 2 \pmod{4}$. 
\end{lemma}

\begin{proof}
Let us parameterize the nonsquare powers of $2$ by: $q_n=2^{2n-1}$, for $n\ge 1$. For each $q_n$ let us denote
$$
2\sqrt{q_n}=2^n\sqrt{2}=m_n+\epsilon_n,\qquad m_n=\lfloor 2\sqrt{q_n}\rfloor,\ \epsilon_n=\{2\sqrt{q_n}\}. 
$$
Since $\sqrt{2}$ is irrational, the elements of the sequence $\epsilon_n$ are pairwise different. Clearly,
$$
\epsilon_n<1/2\ \Longrightarrow \ \epsilon_{n+1}=2\epsilon_n,\quad m_{n+1}=2m_n, 
$$$$
1/2<\epsilon_n\ \Longrightarrow \ \epsilon_{n+1}=\epsilon_n-(1-\epsilon_n),\quad m_{n+1}=2m_n+1.
$$
Suppose $\epsilon_n<1/2$. There exists $r\ge 2$ such that
$2^{-r}<\epsilon_n<2^{-r+1}$; thus, 
\begin{equation}\label{petit}
\epsilon_n<\epsilon_{n+1}<\cdots<\epsilon_{n+r-2}<\dfrac12<\epsilon_{n+r-1}, \quad m_{n+r-1}=2^{r-1}m_n,
\end{equation}
and in particular $m_{n+r}=2^rm_n+1\equiv1\pmod{4}$.
On the other hand, if $1/2<\epsilon_n$, there exists $r\ge 2$ such that
$1-2^{-r+1}<\epsilon_n<1-2^{-r}$; thus, 
\begin{equation}\label{gran}
\epsilon_n>\epsilon_{n+1}>\cdots>\epsilon_{n+r-2}>\dfrac12>\epsilon_{n+r-1}, \quad m_{n+r-1}=2m_{n+r-2}+1,
\end{equation}
and in particular $m_{n+r}=4m_{n+r-2}+2\equiv2\pmod{4}$.

Now, (\ref{petit}) shows that every $\epsilon_n<1/2$ determines some more advanced $\epsilon_{n+r-1}>1/2$, and conversely, (\ref{gran}) shows that every $\epsilon_n>1/2$ determines some more advanced $\epsilon_{n+r-1}<1/2$. Therefore, there are infinitely many $\epsilon_n$ in the interval $(0,1/2)$ and infinitely many in the interval $(1/2,1)$.
In particular, (\ref{petit}) and (\ref{gran}) show respectively that there are infinitely many $n$ with $m_n\equiv1\pmod{4}$ and infinitely many $n$ with   $m_n\equiv2\pmod{4}$.
\end{proof}

Thus, the following result is an immediate consequence of Corollary \ref{optimal}. 

\begin{corollary}
There are infinitely many nonsquare $q=2^n$ such that there is a genus $3$ curve with defect $0$ over $\GF_{q}$.
\end{corollary}

\noindent{\bf Acknowledgments. }It is a pleasure to thank Florian Hess for his enlightening comments on a previous draft of the paper.

\bigskip

\begin{small}
\begin{tabular}{ll}
Enric Nart&\quad Christophe Ritzenthaler\\
Departament de Matem\`atiques,&\quad Institut de Math\'ematiques de Luminy,\\
Universitat Aut\`onoma de Barcelona,&\quad Universit\'e de la Medit\'erran\'ee,\\
08193 Bellaterra, Barcelona, Spain.&\quad 13288 Luminy, Marseille, France.\\
{\tt nart@mat.uab.cat}&\quad {\tt ritzenth@iml.univ-mrs.fr}
\end{tabular}
\end{small}

\end{document}